\documentclass[a4paper,11pt]{article}
\usepackage[utf8]{inputenc}

\usepackage{a4wide}
\usepackage{latexsym,amsfonts,amsmath,amssymb,mathrsfs,url,amsthm}
\usepackage{mathtools}
\usepackage{color,graphicx}
\usepackage{lipsum}
\usepackage{hyperref}
\usepackage{xcolor}

\newtheorem{theorem}{Theorem}
\newtheorem{lemma}{Lemma}
\newtheorem{remark}{Remark}
\newtheorem{proposition}{Proposition}

\newtheorem{conjecture}{Conjecture}

\providecommand{\keywords}[1]
{
  \noindent \small	
  \textbf{Keywords:} #1
}
\providecommand{\amscode}[1]
{
  \noindent \small	
  \textbf{AMS subject classifications:} #1
}

\newcommand{\rd}{\, \mathrm{d}}
\newcommand{\bsx}{\boldsymbol{x}}
\newcommand{\bsy}{\boldsymbol{y}}
\newcommand{\bsz}{\boldsymbol{z}}
\newcommand{\bsgamma}{\boldsymbol{\gamma}}
\newcommand{\EE}{\mathbb{E}}
\newcommand{\NN}{\mathbb{N}}
\newcommand{\PP}{\mathbb{P}}
\newcommand{\RR}{\mathbb{R}}

\newcommand{\Kcal}{\mathcal{K}}

\title{A note on unshifted lattice rules for high-dimensional integration in weighted unanchored Sobolev spaces}
\author{Takashi Goda\thanks{Graduate School of Engineering, The University of Tokyo, 7-3-1 Hongo, Bunkyo-ku, Tokyo 113-8656, Japan ({\tt goda@frcer.t.u-tokyo.ac.jp})}}
\date{\today}

\begin{document}

\maketitle

\begin{abstract}
This short article studies a deterministic quasi-Monte Carlo lattice rule in weighted unanchored Sobolev spaces of smoothness $1$. Building on the error analysis by Kazashi and Sloan, we prove the existence of unshifted rank-1 lattice rules that achieve a worst-case error of $O(n^{-1/4}(\log n)^{1/2})$, with the implied constant independent of the dimension, under certain summability conditions on the weights. Although this convergence rate is inferior to the one achievable for the shifted-averaged root mean squared worst-case error, the result does not rely on random shifting or transformation and holds unconditionally without any conjecture, as assumed by Kazashi and Sloan.
\end{abstract}


\amscode{65C05, 65D30, 65D32}

\section{Introduction}\label{sec:intro}
\sloppy

We study numerical integration of functions defined over the multi-dimensional unit cube $[0,1)^d$ with $d\in \NN$. For an integrable function $f: [0,1)^d\to \RR$, we denote its integral by
\[ I_d(f):=\int_{[0,1)^d}f(\bsx)\rd \bsx. \]
We consider approximating $I_d(f)$ by a deterministic quasi-Monte Carlo rank-1 lattice rule \cite{DKP22,SJ94}. That is, for a given number of points $n$ and a generating vector $\bsz\in \{1,\ldots,n-1\}^d$, we define the approximation
\[ Q_{d,n,\bsz}(f) := \frac{1}{n}\sum_{i=0}^{n-1}f(\bsx_i), \]
where the integration nodes $\bsx_i\in [0,1)^d$ are given by
\[ \bsx_i=\left( \left\{ \frac{iz_1}{n}\right\},\ldots, \left\{ \frac{iz_d}{n}\right\} \right), \]
and $\{x\}=x-\lfloor x\rfloor$ denotes the fractional part of $x\ge 0$.

We assume that the integrand $f$ belongs to the weighted unanchored Sobolev space of smoothness~$1$, denoted by $H_{d,\bsgamma}$, where $\bsgamma=(\gamma_u)_{u\subset \NN}$ is a collection of non-negative weights $\gamma_u\geq 0$  representing the relative importance of variable subsets \cite{SW98}. This space consists of functions whose first-order mixed partial derivatives are square-integrable. Moreover, it is a reproducing kernel Hilbert space with reproducing kernel
\[ K_{d,\bsgamma}(\bsx,\bsy)=\sum_{u\subseteq \{1,\ldots,d\}}\gamma_u \prod_{j\in u}\eta(x_j,y_j),\]
for $\bsx,\bsy\in [0,1)^d$, where
\[ \eta(x,y)=\frac{1}{2}B_2(|x-y|)+B_1(x)B_1(y),\]
and $B_i$ denotes the Bernoulli polynomial of degree $i$. As a quality criterion, we consider the worst-case error
\[ e(n,\bsz):=\sup_{\substack{f\in H_{d,\bsgamma}\\ \|f\|_{d,\bsgamma}\le 1}}\left| I_d(f)-Q_{d,n,\bsz}(f)\right|, \]
where $\|f\|_{d,\bsgamma}$ denotes the norm of $f$ in the space $H_{d,\bsgamma}$. We refer to \cite[Chapter~7.1]{DKP22} for the precise definitions of the inner product and the norm in $H_{d,\bsgamma}$.

Although it is known that a worst-case error of $O(n^{-1+\varepsilon})$ for arbitrarily small $\varepsilon>0$ can be achieved by suitably designed integration rules using $n$ function evaluations, existing results for rank-1 lattice rules rely on applying random shifts \cite{Kuo03} or transformations \cite{DNP14,GSY19} to the integration nodes. 

An exception is the work of Kazashi and Sloan \cite{KS18}, who studied the worst-case error of \emph{unshifted} rank-1 lattice rules. Their approach was to use an averaging argument to prove the existence of a good generating vector $\bsz$ such that the worst-case error is small. Specifically, they considered the average of the squared worst-case error over all generating vectors $\bsz \in \{1,\ldots,n-1\}^d$:
\[ \overline{e}^2(n)=\frac{1}{(n-1)^d}\sum_{\bsz\in \{1,\ldots,n-1\}^d}e^2(n,\bsz). \]
By combining Equation~(11), Proposition~3, and Lemma~4 in \cite{KS18}, we obtain the following bound on $\overline{e}^2(n)$.

\begin{proposition}[\emph{Kazashi and Sloan} \cite{KS18}]\label{prop:KS18}
    Let $n$ be an odd prime. Then the mean square worst-case error satisfies
    \[ \overline{e}^2(n)\leq \frac{1}{n}\sum_{\emptyset \neq u\subseteq \{1,\ldots,d\}}\gamma_u \left[ c_u+\left(\frac{1}{2\pi^2} \frac{n}{n-1} \right)^{|u|}\sum_{\kappa=1}^{n-1}\left( T_{n}(\kappa)+\frac{10\pi^2 \log n}{9n}\right)^{|u|}\right] , \]
    where 
    \[ c_u :=\frac{2}{3^{|u|}}+\frac{1}{4^{|u|}}, \quad \text{and}\quad T_n(\kappa) := \sum_{q=1}^{(n-1)/2}\frac{1}{q\, |r(q\kappa,n)|}, \]
    with $r(j, n)$ denoting the unique integer congruent to $j$ modulo $n$ in the set $\{-(n-1)/2,\ldots,(n-1)/2\}$. That is,
    \[ r(j,n) := \begin{cases} j\bmod n & \text{if $j\bmod n \leq (n-1)/2$},\\ (j \bmod n)-n & \text{if $j\bmod n > (n-1)/2$}.\end{cases} \]
\end{proposition}

The remaining issue is to give an upper bound on $T_n(\kappa)$ for $1\leq \kappa\leq n-1$. Although \cite[Lemma~4]{KS18} shows that $T_n(\kappa)\le \pi^2/6$ uniformly for all $\kappa$, substituting this constant bound into Proposition~\ref{prop:KS18} yields an upper bound on $\overline{e}^2(n)$ that does not decay as $n \to \infty$. To address this, Kazashi and Sloan proposed a number-theoretic conjecture, which is rephrased as follows. Under this assumption, they showed that $\overline{e}^2(n)$ can be bounded by $O(1/n)$, up to a dimension-independent logarithmic factor.

\begin{conjecture}[\emph{Kazashi and Sloan} \cite{KS18}]\label{conjecture:KS18}
    Let $n$ be an odd prime. There exist constants $C_1,C_2>0$ and $\alpha\ge 2$, all independent of $n$, such that
    \[ T_n(\kappa)> C_1\frac{(\log n)^{\alpha}}{n}\]
    holds for at most $C_2 (\log n)^{\alpha}$ values of $\kappa$ among $\{1,\ldots,n-1\}$. 
\end{conjecture}

We now state the aim of this article. First, we prove that Conjecture~\ref{conjecture:KS18} does not hold. As a remedy, we then establish a weaker result regarding the quantity $T_n(\kappa)$. This leads to an upper bound on $\overline{e}^2(n)$ of $O(n^{-1/2} \log n)$, which in turn implies the existence of a good generating vector $\bsz$, such that the corresponding unshifted rank-1 lattice rule $Q_{d,n,\bsz}$ achieves a worst-case error of $O(n^{-1/4} (\log n)^{1/2})$. Although this rate is far from optimal, it provides---so far as the author is aware---the first theoretical evidence that unshifted rank-1 lattice rules can still be effective for non-periodic functions in $H_{d,\bsgamma}$. Whether this rate can be improved remains an open question.

\section{Results}

The first result is as follows:
\begin{theorem}\label{thm:first}
Conjecture~\ref{conjecture:KS18} does not hold. 
\end{theorem}

\begin{proof}
Assume $n\ge 7$, which ensures that $(n-1)/2\ge \sqrt{n}$. For any $1\leq \kappa\leq \lfloor\sqrt{n}\rfloor$, consider the term with $q=1$ in the definition of $T_n(\kappa)$. We have
\begin{align*}
    T_n(\kappa)\geq \frac{1}{|r(\kappa,n)|}=\frac{1}{\kappa}\geq \frac{1}{\sqrt{n}}.
\end{align*}
Now fix any constants $C_1,C_2>0$ and $\alpha\ge 2$. Then there exists $n_0\in \NN$ such that, for all $n\geq n_0$, we have
\[ \lfloor \sqrt{n}\rfloor \geq C_2 (\log n)^{\alpha}\quad \text{and}\quad \frac{1}{\sqrt{n}}\ge C_1\frac{(\log n)^{\alpha}}{n}. \]
It follows that, for these $n$,
\[ T_n(\kappa) \ge C_1\frac{(\log n)^{\alpha}}{n}, \quad \text{for all $1\leq \kappa\leq \lfloor\sqrt{n}\rfloor$}.\]
Hence,
\[ \left| \left\{ 1\leq \kappa\leq n-1\mid T_n(\kappa) \ge C_1\frac{(\log n)^{\alpha}}{n} \right\}\right|\geq \lfloor \sqrt{n}\rfloor \geq C_2 (\log n)^{\alpha}, \]
which contradicts the existence of constants $C_1, C_2 > 0$ and $\alpha \ge 2$ such that Conjecture~\ref{conjecture:KS18} holds. 
\end{proof}

We now establish the aforementioned weaker result for the quantity $T_n(\kappa)$, which will play a central role in deriving our bound on the mean square worst-case error $\overline{e}^2(n)$.

\begin{lemma}\label{lem}
    Let $n$ be an odd prime. Then the inequality
    \[ T_n(\kappa)\ge 4\frac{\log n}{\sqrt{n}}\]
    holds for at most $4 \sqrt{n}\log n$ values of $\kappa\in \{1,\ldots,n-1\}$. 
\end{lemma}

\begin{remark}
    It can be inferred from the proof of Theorem~\ref{thm:first} that the inequality $T_n(\kappa)\ge 1/\sqrt{n}$ is satisfied for \emph{at least} $\lfloor\sqrt{n}\rfloor$ values of $\kappa\in \{1,\ldots,n-1\}$. This implies that the result of Lemma~\ref{lem} is essentially optimal, up to a logarithmic factor in $n$.
\end{remark}

\begin{proof}[Proof of Lemma~\ref{lem}]
    Throughout this proof, let $X$ be a uniformly distributed random variable over $\{1,\ldots,n-1\}$. Then 
    \[ \mu := \EE[T_n(X)]=\sum_{q=1}^{(n-1)/2}\frac{1}{q}\EE\left[\frac{1}{|r(qX,n)|}\right].\]
    For any fixed $q\in \{1,\ldots,(n-1)/2\}$, the map $\kappa \mapsto r(q\kappa,n)$ defines a bijection from $\{1,\ldots,n-1\}$ to $\{-(n-1)/2,\ldots,(n-1)/2\} \setminus \{0\}$. Hence,
    \begin{align*}
    \mathbb{E}\left[ \frac{1}{|r(qX,n)|} \right] 
    &= \frac{1}{n-1} \sum_{j=1}^{n-1} \frac{1}{|r(j,n)|} 
    = \frac{2}{n-1} \sum_{j=1}^{(n-1)/2} \frac{1}{j} \\
    &\le \frac{2}{n-1} \left( 1 + \int_1^{(n-1)/2} \frac{1}{x} \, \mathrm{d}x \right) 
    = \frac{2}{n-1} \left( 1 + \log \frac{n-1}{2} \right).
    \end{align*}

    Therefore,
    \begin{align*}
    \mu &\le \frac{2}{n-1} \left( 1 + \log \frac{n-1}{2} \right) \sum_{q=1}^{(n-1)/2} \frac{1}{q} \\
    &\le \frac{2}{n-1} \left( 1 + \log \frac{n-1}{2} \right)^2 
    \le 16 \frac{(\log n)^2}{n} =: \tilde{\mu}.
    \end{align*}

    Since $T_n(X) > 0$, applying Markov's inequality gives
    \[
    \PP\left[ T_n(X) \ge t \right] \le \frac{\mu}{t},
    \]
    for any $t > 0$. Setting $t = \tilde{\mu}^{1/2}$, we obtain
    \[
    \PP\left[ T_n(X) \ge \tilde{\mu}^{1/2} \right] \le \frac{\mu}{\tilde{\mu}^{1/2}} \le \tilde{\mu}^{1/2}.
    \]
    Thus, the number of $\kappa \in \{1,\ldots,n-1\}$ satisfying
    \[
    T_n(\kappa) \ge \tilde{\mu}^{1/2} = \frac{4 \log n}{\sqrt{n}}
    \]
    is at most $(n-1)\tilde{\mu}^{1/2} \le 4 \sqrt{n} \log n$.
\end{proof}

Finally, combining Proposition~\ref{prop:KS18} and Lemma~\ref{lem}, we obtain a bound $\overline{e}^2(n) = O(n^{-1/2} \log n)$ as follows.

\begin{theorem}\label{thm:final}
    Let $n$ be an odd prime. Then the mean squared worst-case error $\overline{e}^2(n)$ of unshifted rank-1 lattice rules in the space $H_{d,\bsgamma}$ satisfies
    \[ \overline{e}^2(n) \leq \frac{\log n}{\sqrt{n}}\sum_{\emptyset \neq u\subseteq \{1,\ldots,d\}}\gamma_u C_u, \]
    where
    \[ C_u := \frac{2}{3^{|u|}}+\frac{1}{4^{|u|}}+4\left(\frac{23}{24}\right)^{|u|} + \left( \frac{3}{\pi^2}+\frac{5}{6}\right)^{|u|}.\]
\end{theorem}

\begin{proof}
    We aim to bound the right-hand side of the inequality in Proposition~\ref{prop:KS18}. Define
    \[ \Kcal_n := \left\{ 1\leq \kappa\leq n-1\mid T_n(\kappa)\ge 4\frac{\log n}{\sqrt{n}} \right\}. \]
    It follows from Lemma~\ref{lem} that $|\Kcal_n| \leq 4 \sqrt{n}\log n$. For any $\kappa\in \Kcal_n$, we use a constant bound $T_n(\kappa)\le \pi^2/6$ from \cite[Lemma~4]{KS18} to get
    \[ T_{n}(\kappa)+\frac{10\pi^2 \log n}{9n} \le \frac{\pi^2}{6}+\frac{10\pi^2}{9}=\frac{23\pi^2}{18}.\]
    For any $\kappa\notin \Kcal_n$, using Lemma~\ref{lem}, we have
    \[ T_{n}(\kappa)+\frac{10\pi^2 \log n}{9n}< 4\frac{\log n}{\sqrt{n}}+\frac{10\pi^2 \log n}{9\sqrt{n}}=\left( 4+\frac{10\pi^2}{9}\right)\frac{\log n}{\sqrt{n}}. \]

    Then, for any non-empty subset $u\subseteq \{1,\ldots,d\}$, it holds that
    \begin{align*}
        & \sum_{\kappa=1}^{n-1}\left( T_{n}(\kappa)+\frac{10\pi^2 \log n}{9n}\right)^{|u|} \\
        & = \sum_{\kappa\in \Kcal_n}\left( T_{n}(\kappa)+\frac{10\pi^2 \log n}{9n}\right)^{|u|}+\sum_{\kappa\notin \Kcal_n}\left( T_{n}(\kappa)+\frac{10\pi^2 \log n}{9n}\right)^{|u|}\\
        & \le |\Kcal_n| \left(\frac{23\pi^2}{18}\right)^{|u|} + (n-1)\left( 4+\frac{10\pi^2}{9}\right)^{|u|}\left(\frac{\log n}{\sqrt{n}}\right)^{|u|}\\
        & \le \sqrt{n}\log n \left[ 4\left(\frac{23\pi^2}{18}\right)^{|u|} + \left( 4+\frac{10\pi^2}{9}\right)^{|u|}\right] =: \tilde{c}_u \sqrt{n}\log n,
    \end{align*}
    with 
    \[ \tilde{c}_u = 4\left(\frac{23\pi^2}{18}\right)^{|u|} + \left( 4+\frac{10\pi^2}{9}\right)^{|u|}. \]
    This bound, applied to the inequality in Proposition~\ref{prop:KS18}, leads to
    \begin{align*}
        \overline{e}^2(n) & \leq \frac{1}{n}\sum_{\emptyset \neq u\subseteq \{1,\ldots,d\}}\gamma_u \left[ c_u+\left(\frac{1}{2\pi^2} \frac{n}{n-1} \right)^{|u|}\tilde{c}_u \sqrt{n}\log n\right]\\
        & \leq \frac{\sqrt{n}\log n}{n}\sum_{\emptyset \neq u\subseteq \{1,\ldots,d\}}\gamma_u \left[ c_u+\left(\frac{3}{4\pi^2}\right)^{|u|}\tilde{c}_u \right].
    \end{align*}
    This establishes the claimed bound in Theorem~\ref{thm:final}, completing the proof.
\end{proof}

\begin{remark}
    Theorem~\ref{thm:final} implies the existence of a good generating vector $\bsz$ whose worst-case error in the space $H_{d,\bsgamma}$ satisfies
    \[
        e^2(n,\bsz) \leq \left(\frac{\log n}{\sqrt{n}} \sum_{\emptyset \neq u\subseteq \{1,\ldots,d\}} \gamma_u C_u \right)^{1/2}.
    \]
    This error bound is independent of the dimension $d$ provided that
    \[
        C := \sum_{|u|\leq \infty} \gamma_u C_u < \infty.
    \]
    Although we omit the details, in the case of product weights, that is, $\gamma_u = \prod_{j\in u} \gamma_j$ for a sequence $\gamma_1, \gamma_2, \ldots \in \RR_{\geq 0}$, this condition simplifies to
    \[
        \sum_{j=1}^\infty \gamma_j < \infty.
    \]
\end{remark}

\section*{Acknowledgments}
The author would like to thank Yoshihito Kazashi for valuable discussions. This work was supported by JSPS KAKENHI Grant Number 23K03210.

\bibliographystyle{amsplain} 
\bibliography{ref}

\end{document}